\documentclass[a4paper,12pt]{article}
\usepackage{amssymb,amsmath,amsthm,latexsym}
\usepackage{amsfonts}
	\usepackage{amsfonts}
\usepackage{graphicx}
\usepackage[pdftex,bookmarks,colorlinks=false]{hyperref}
\usepackage{verbatim}
\usepackage{caption}
\usepackage{subcaption}
\usepackage[hmargin=1in,vmargin=1in]{geometry}

\newtheorem{theorem}{Theorem}[section]

\newtheorem{corollary}[theorem] {Corollary}
\newtheorem{definition}[theorem]{Definition}

\newtheorem{lemma} [theorem]{Lemma}

\newtheorem{proposition}[theorem]{Proposition}
\newtheorem{remark}[theorem]{Remark}

\title{\bf Weak Integer Additive Set-Indexers of Certain Graph Operations}
\author{{\bf N K Sudev \footnote{Department of Mathematics, Vidya Academy of Science \& Technology, Thalakkottukara, Thrissur-680501,Kerala, India. email: {\em sudevnk@gmail.com}, Phone:{+919497557876.}}     and    {\bf K A Germina\footnote{Department of Mathematics, School of Mathematical \& Physical Sciences, Central University of Kerala, Kasaragod-671531, Kerala, India email:{\em srgerminaka@gmail.com}, Phone:{+919744859390.}}}}}
\date{}

\begin{document}
\maketitle

\begin{abstract}
An integer additive set-indexer is defined as an injective function $f:V(G)\rightarrow 2^{\mathbb{N}_0}$ such that the induced function $g_f:E(G) \rightarrow 2^{\mathbb{N}_0}$ defined by $g_f (uv) = f(u)+ f(v)$ is also injective, where $f(u)+f(v)$ is the sum set of $f(u)$ and $f(v)$ and $\mathbb{N}_0$ is the set of all non-negative integers. If $g_f(uv)=k~\forall~uv\in E(G)$, then $f$ is said to be a $k$-uniform integer additive set-indexers. An integer additive set-indexer $f$ is said to be a weak integer additive set-indexer if $|g_f(uv)|=max(|f(u)|,|f(v)|)~\forall ~ uv\in E(G)$. We have some  characteristics of the graphs which admit weak integer additive set-indexers. In this paper, we study the admissibility of weak integer additive set-indexer by certain finite graph operations.
\end{abstract}
\textbf{Key words}: Integer additive set-indexers, weak integer additive set-indexers, mono-indexed elements of a graph, sparing number of a graph.\\
\textbf{AMS Subject Classification: 05C78} 

\section{Introduction}

For all  terms and definitions, not defined specifically in this paper, we refer to \cite{FH}, \cite{BM1}, and \cite{ND}. Unless mentioned otherwise, all graphs considered here are simple, finite and have no isolated vertices.

An {\em integer additive set-indexer} (IASI, in short) is defined in \cite{GA} as an injective function $f:V(G)\rightarrow 2^{\mathbb{N}_0}$ such that the induced function $g_f:E(G) \rightarrow 2^{\mathbb{N}_0}$ defined by $g_f (uv) = f(u)+ f(v)$, where $f(u)+f(v)$ is the sumeset of $f(u)$ and $f(v)$, is also injective. If $g_f(e)=k~\forall~ e\in E(G)$, then $f$ is called a $k$-uniform IASI.

The cardinality of the labeling set of an element (vertex or edge) of a graph $G$ is called the {\em set-indexing number} of that element.

The characteristics of weak IASI graphs have been done in \cite{GS1} and \cite{GS2}. The following are the major notions and results established in these papers.

\begin{lemma}\label{L-Card}
\cite{GS1} For an integer additive set-indexer $f$ of a graph $G$, we have $max(|f(u)|,|f(v)|)\le |g_f(uv)|= |f(u)+f(v)| \le |f(u)| |f(v)|$, where $u,v\in V(G)$.
\end{lemma}

\begin{definition}{\rm
\cite{GS1} An IASI $f$ is said to be a {\em weak IASI} if $|g_f(uv)|=max(|f(u)|,|f(v)|)$ for all $u,v\in V(G)$. A graph which admits a weak IASI may be called a {\em weak IASI graph}.}
\end{definition}

\begin{definition}{\rm
\cite{GS2} An element (a vertex or an edge) of graph which has the set-indexing number 1 is called a {\em mono-indexed element} of that graph. The {\em sparing number} of a graph $G$ is defined to be the minimum number of mono-indexed edges required for $G$ to admit a weak IASI and is denoted by $\varphi(G)$.}
\end{definition}

\begin{theorem}\label{T-WSG}
\cite{GS3} If a graph $G$ is a weak IASI graph, then any subgraph $H$ of $G$ is also a weak IASI graph. Or equivalently, if $G$ is a graph which has no weak IASI, then any supergraph of $G$ does not have a weak IASI.  
\end{theorem}

\begin{theorem}\label{T-SB1}
\cite{GS3}  If a connected graph $G$ admits a weak IASI, then $G$ is bipartite or $G$ has at least one mono-indexed edge. Hence, all paths, trees and even cycles admit a weak IASI. We observe that the sparing number of bipartite graphs is $0$.
\end{theorem}

\begin{theorem}\label{T-WKN}
\cite{GS3}  The complete graph $K_n$ admits a weak IASI if and only if the number of edges of $K_n$ that have set-indexing number $1$ is $\frac{1}{2}(n-1)(n-2)$.
\end{theorem}

\begin{theorem}\label{T-WUOC}
\cite{GS3}  An odd cycle $C_n$ has a weak IASI if and only if it has at least one mono-indexed edge. 
\end{theorem}

\begin{theorem}\label{T-NME}
\cite{GS3}  Let $C_n$ be a cycle of length $n$ which admits a weak IASI, for a positive integer $n$. Then, $C_n$ has an odd number of mono-indexed edges when it is an odd cycle and has even number of mono-indexed edges, when it is an even cycle. 
\end{theorem}  

\section{Weak IASI of Graph Operations}

In this section, we discuss the admissibility of weak IASI to certain operations of graphs. 

In fact, the intersection of paths or cycles or both is a path and hence by Remark \ref{T-SB1}, it admits a weak IASI. For finite number of cycles $C_{n_1},C_{n_2},C_{n_3}, \ldots,C_{n_r}$ which admit weak IASIs, their intersection $\displaystyle{\bigcap_{i=1}^{r}C_{n_i}}$ admits a weak IASI if all cycles $C_{n_i}$ have a common path.

Given two graphs $G_1$ and $G_2$, the intersection $G_1\cap G_2$ need not be a path. If $G_1$ and $G_2$ admit weak IASIs, say $f_1$ and $f_2$ respectively, then their intersection $G_1\cap G_2$ admits a weak IASI if and only if  $f_1$ and $f_2$ are suitably defined in such a way that $f=f_1|_{G_1\cap G_2}=f_2|_{G_1\cap G_2}$, where $f_i|_{G_1\cap G_2}, i=1,2$, is the restriction of $f_i$ to $G_1\cap G_2$.

\subsection{Weak IASI of the Union of Graphs}

\begin{definition}{\rm
\cite{BM1} The union $G_1\cup G_2$ of two graphs (or two subgraphs of a given graph) $G_1(V_1,E_1)$ and $G_2(V_2,E_2)$ is the graph whose vertex set is $V_1\cup V_2$ and the edge set is $E_1\cup E_2$. If $G_1$ and $G_2$ are disjoint graphs, then their union is called {\em disjoint union} of $G_1$ and $G_2$.} 
\end{definition}

The union of two graphs we mention here need not be the disjoint union. First, we discuss the admissibility of weak IASI by the union of two graphs $G_1$ and $G_2$.

\begin{theorem}\label{T-WUG}
Let $G_1$ and $G_2$ be two cycles. Then, $G_1\cup G_2$ admits a weak IASI if and only if both $G_1$ and $G_2$ are weak IASI graphs.
\end{theorem}
\begin{proof}
 Let $G_1$ and $G_2$ be two weak IASI graphs. Let $f_1:V(G_1)\to 2^{\mathbb{N}_0}$ be a weak IASI for $G_1$ and $f_2:V(G_2)\to 2^{\mathbb{N}_0}$ be a weak IASI for $G_2$. 
 
If $G_1$ and $G_2$ be two disjoint cycles, then define $f:V(G_1\cup G_2)\to 2^{\mathbb{N}_0}$  by \[ f(v) = \left\{
  \begin{array}{l l}
    f_1(v) & \quad \text{if $v \in G_1$}\\
    f_2(v) & \quad \text{if $v\in G_2$}
  \end{array} \right.\].

If $G_1$ and $G_2$ be two graphs with some common elements, then define $f$ as above, with an additional condition that $f_1=f_2=f$ for all the elements in $G_1\cap G_2$. Therefore, $f$ is a weak IASI for $G_1\cup G_2$.

Conversely, assume that $G_1\cup G_2$ is a weak IASI graph. Then, both $G_1$ and $G_2$ are subgraphs of $G_1\cup G_2$. Hence, by Theorem \ref{T-WSG}, both $G_1$ and $G_2$ admit weak IASIs.
\end{proof}

In the following theorem we discuss about the sparing number of the union of two weak IASI graphs. 

\begin{theorem}\label{T-SNUG}
Let $G_1$ and $G_2$ be two weak IASI graphs. Then, $\varphi(G_1\cup G_2) = \varphi(G_1)+\varphi(G_2)-\varphi(G_1\cap G_2)$.
\end{theorem}
\begin{proof}
Let $G_1$ and $G_2$ be two weak IASI graphs with the corresponding weak IASIs $f_1$ and $f_2$ respectively.  Define a function $f: G_1\cup G_2\to 2^{\mathbb{N}_0}$, such that
\[ f(v)= \left\{
\begin{array}{l l}	
f_1(v)& \quad \text{if $v\in G_1$}\\
f_2(v)& \quad \text{if $v\in G_2$}
\end{array} \right.\]
Then, 
\begin{eqnarray*}
G_1\cup G_2 & = & (G_1-(G_1\cap G_2))\cup (G_2-(G_1\cap G_2))\cup (G_1\cap G_2)\\
\varphi(G_1\cup G_2) & = & \varphi(G_1)-\varphi(G_1\cap G_2)+ \varphi(G_2)-\varphi(G_1\cap G_2)+ \varphi(G_1\cap G_2)\\
& = & \varphi(G_1)+\varphi(G_2)-\varphi(G_1\cap G_2). 
\end{eqnarray*}
This completes the proof.
\end{proof}

\subsection{Weak IASI of the Join of Graphs}

\begin{definition}{\rm
\cite{FH} Let $G_1(V_1,E_1)$ and $G_2(V_2,E_2)$ be two graphs. Then, their {\em join} (or {\em sum}), denoted by $G_1+G_2$, is the graph whose vertex set is $V_1\cup V_2$ and edge set is $E_1\cup E_2\cup E_{ij}$, where $E_{ij}=\{u_iv_j:u_i\in G_1,v_j\in G_2\}$. }
\end{definition}

In this section, we verify the admissibility of a weak IASI by the join of paths, cycles and graphs. We proceed by using the following notion.

The graph $P_n+K_1$ is called a {\em fan graph} and is denoted by $F_{1,n}$. The following result establishes the admissibility of weak IASI by a fan graph $F_{1,n}$.

\begin{theorem}\label{T-PnK1}
Let $F_{1,n}=P_n+K_1$. Then, $F_{1,n}$ admits a weak IASI if and only if $P_n$ is $1$-uniform or $K_1$ is mono-indexed.
\end{theorem}
\begin{proof}
Assume that $P_n$ is $1$-uniform. Denote the single vertex in $K_1$ by $v$. If we label $v$ by a singleton set, then $F_{1,n}$ is $1$-uniform. If label $v$ by a non-singleton vertex, then every edge in $F_{1,n}$ has at least one mono-indexed vertex. This labeling is a weak IASI for $F_{1,n}$. Assume that $P_n$ is not $1$-uniform. If $K_1$ is mono-indexed, then the corresponding set-label in $F_{1,n}$ is a weak IASI.

Conversely, assume that $F_{1,n}$ is a weak IASI graph. If $P_n$ is $1$-uniform, then the proof is complete. Hence, assume that $P_n$ is not $1$-uniform. Then, at least one vertex of $P_n$ must have a non-singleton set-label. Therefore, $K_1$ must be mono-indexed, since $F_{1,n}$ admits a weak IASI.

This completes the proof.
\end{proof}

From Theorem \ref{T-PnK1}, we note that the number of mono-indexed edges in $F_{1,n}$ is minimum when $K_1$ is mono-indexed. Hence, we have the following result.

\begin{proposition}
The sparing number of a fan graph $F_{1,n}=P_n+K_1$ is $\lceil \frac{n-1}{2} \rceil$.
\end{proposition}

The following theorem establishes the admissibility of the join of two paths in a given graph $G$.

\begin{theorem}\label{T-WP+}
Let $P_m,P_n$ be two paths. Then, the join $P_m+P_n$ admits a weak IASI if and only if $P_m$ or $P_n$ is a $1$-uniform graph.
\end{theorem}
\begin{proof}
By Remark \ref{T-SB1}, all paths admit weak IASI. Without loss of generality, assume that $P_m$ is $1$-uniform. Let $E_{ij}=\{uv: u\in P_m,v\in P_n\}$. Then, every edge in $E_{ij}$ has at least one mono-indexed vertex. Hence, $P_m+P_n$ admits a weak IASI. Conversely, assume that the join $P_m+P_n$ of two paths $P_m$ and $P_n$ admits a weak IASI. If $P_m$ and $P_n$ are not $1$-uniform, then neither of the end vertices of some edge $e$ in $E_{ij}$ are mono-indexed, which is a contradiction to the hypothesis. Hence, either $P_m$ or $P_n$ must be $1$-uniform.
\end{proof}

\begin{definition}{\rm
A {\em wheel graph} $W_n$ is a graph with $n$ vertices, $(n\ge 4)$, formed by connecting  all vertices of an $(n-1)$-cycle $C_{n-1}$ to a single vertex other than the vertices of $C_{n-1}$. That is, $W_n=K_1+C_{n-1}$. }
\end{definition}

\begin{theorem}\label{T-WWG}
Let $C_n$ be a cycle of length $n$ which has a weak IASI. Then, the wheel graph $W_{n+1}=C_n+K_1$ admits a weak IASI if and only if $C_n$ is $1$-uniform or $K_1$ is mono-indexed.  
\end{theorem}
\begin{proof}
Let $v$ be the single vertex of $K_1$. If we label $v$ by a non-singleton set, then, as $W_{n+1}$ has a weak IASI, no vertex of $C_n$ can have a non-singleton set-label. That is, $C_n$ is $1$-uniform. Conversely, if $C_n$ is $1$-uniform, then $W_{n+1}$ is a weak IASI graph for any set-label of $K_1$.

Next, assume that $C_n$ is not $1$-uniform. Let $W_{n+1}$ is a weak IASI graph. Since $v$ is adjacent to every vertex of $C_n$, it can only have a singleton set-label. That is, $K_1$ is mono-indexed.  Conversely, If we label $v$ by a singleton set, then, since $C_n$ has a weak IASI, it forms a weak IASI for $W_{n+1}$.

Hence, the wheel graph $W_{n+1}=C_n+K_1$ admits a weak IASI if and only if $C_n$ is $1$-uniform or $K_1$ is mono-indexed.
\end{proof}

From Theorem \ref{T-WWG}, we note that the number of mono-indexed edges in $W_{n+1}$ is minimum when $K_1$ is mono-indexed. Hence, we have the following proposition.

\begin{proposition}
The sparing number of a wheel graph $W_{n+1}=C_n+K_1$ is $\lceil \frac{n-1}{2} \rceil$.
\end{proposition}

\begin{theorem}\label{T-WC+P}
Let $C_n$ be a cycle that admits a weak IASI and $P_m$ be a path. Then, their join $G=C_n+P_m$ admits a weak IASI if and only if either $C_n$ or $P_m$ is a $1$-uniform IASI graph.
\end{theorem}
\begin{proof}
First, assume that either $C_n$ or $P_m$ is a $1$-uniform IASI graph. Then, every edge $u_iv_j$ in $G=C_n+P_m$, where $u_i\in P$ and $v_j\in C_n$ has at least one mono-indexed end vertex. Then, such a labeling is a weak IASI for $G$.

Conversely, assume that $G=C_n+P_m$ admits a weak IASI. If possible, assume that neither $C_n$ nor $P_m$ is $1$-uniform. Let $u_i$ be a vertex in $P_m$ and $v_j$ be a vertex in $C_n$ which have set-indexing numbers greater than $1$. Since every vertex of $P_m$ is adjacent to every vertex of $C_n$ in $G$, we have an edge $u_iv_j$ in $G$ whose both the end vertices have set-indexing number greater than $1$, which is a contradiction to the hypothesis. Therefore, either $P_m$ or $C_n$ must be $1$-uniform. 
\end{proof}

\begin{theorem}\label{T-WC+}
Let $C_m$ and $C_n$ be two cycles which admit weak IASIs. Then $C_m+C_n$ admits a weak IASI if and only if all elements of either $C_m$ or $C_n$ are mono-indexed. In other words,the join $C_m+C_n$ of two weak IASI cycles $C_m$ and $C_n$, admits a weak IASI if and only if either $C_m$ or $C_n$ is a $1$-uniform IASI graph.
\end{theorem}
\begin{proof}
Without loss of generality, let all elements of the cycle $C_m$ are mono-indexed. Also, let the cycle $C_n$ admits a weak IASI.  Then, every edge in $C_m+C_n$ has at least one mono-indexed end vertex. Therefore, $C_m+C_n$ admits a weak IASI. 

Conversely, Assume that $C_m+C_n$ admits a weak IASI. If possible, assume that there exist some elements (vertices or edges) in both $C_m$ and $C_n$ which are not mono-indexed. Let $u_i$ be a vertex in $C_m$ and $v_j$ be a vertex in $C_n$ which are not mono-indexed. Then, the edge $u_iv_j$ in $C_m+C_n$ has both the end vertices having set-indexing number greater than $1$, which is a contradiction to the hypothesis. Hence, either $C_m$ or $C_n$ must have all its elements mono-indexed.
\end{proof}

The following result is a more general result of the above theorems.

\begin{theorem}
Let $G_1$ and $G_2$ be two weak IASI graphs. Then, the graph $G_1+G_2$ is a weak IASI graph if and only if either $G_1$ or $G_2$ is a $1$-uniform IASI graph.
\end{theorem}

In fact, we can generalise Theorem \ref{T-WC+}, to the join of finite number of cycles as given in the following theorem.

\begin{theorem}\label{T-WC++}
Let  $C_{n_1},C_{n_2},C_{n_3},\ldots,C_{n_r}$ be $r$ cycles which admit weak IASIs.  Then, their join $\displaystyle{\sum_{i=1}^{r}C_{n_i}}$ admits a weak IASI if and only if all cycles $C_{n_i}$, except one, are $1$-uniform IASI graphs.
\end{theorem}

\begin{proof}
Let $G=\displaystyle{\sum_{i=1}^{r}}C_{n_i}$. Without loss of generality, assume that all cycles, except $C_1$, are $1$-uniform. Then, all edges in the graph $G$ have at least one mono-indexed end vertex . That is, $G$ is a weak IASI graph.

Conversely, $C$ is a weak IASI Graph. Since every vertex of each cycle is adjacent to the vertices of all other cycles, the vertices of $C$ that are not mono-indexed must belong to the same cycle. Therefore, all cycles in $C$, except one, are $1$-uniform.
\end{proof}

Furthermore, we observe that Theorem \ref{T-WC++} is true not only for finite cycles in a given graph $G$, but for finite number of graphs too. Hence, we propose the following result.

\begin{theorem}\label{T-WG++}
Let $G_1,G_2,G_3,\ldots\ldots, G_n$ be weak IASI graphs. Then, the graph $\displaystyle{\sum_{i=1}^{n}G_i}$ is a weak IASI graph if and only of all given graphs $G_i$, except one, are $1$-uniform IASI graphs.
\end{theorem}

Admissibility of weak IASI by graph joins have been discussed so far. Now, we proceed to discuss about the sparing number of these graphs. The following results provide the sparing number of the join of two paths or cycles which admit weak IASI. 

\begin{proposition}
Let $P_m$ and $P_n$ be two paths, where $m < n$. Then, the sparing number of $G$ is given by 
\[\varphi(P_m+P_n) = \left\{
\begin{array}{l l}	
\frac{m}{2}(n+2)& \quad \text{if $P_n$ is even.}\\
\frac{m}{2}(n+1)& \quad \text{if $P_n$ is odd.}
\end{array} \right.\]
\end{proposition}
\begin{proof}
Let $P_m$ and $P_n$ be two paths of lengths $m$ and $n$ respectively. Let $m < n$. By Theorem \ref{T-WG++}, $P_m+P_n$ is a weak IASI graph if and only if either $P_m$ or $P_n$ is $1$-uniform. Since $m < n$, let $P_m$ be $1$-uniform. 

Let $P_n$ be of even length. Then, $P_n$ has $\frac{n}{2}$ mono-indexed edges connecting $P_m$ and $P_n$. Therefore, there are $m.\frac{n}{2}$ mono-indexed edges. Hence, the total number of mono-indexed edges is $m+m.\frac{n}{2} = \frac{m}{2}(n+2)$. 

Let Let $P_n$ be of odd length. Therefore, $P_n$ has $\frac{(n-1)}{2}$ mono-indexed edges connecting $P_m$ and $P_n$. Hence, the total number of mono-indexed edges is $m+m.\frac{n-1}{2} = \frac{m}{2}(n+1)$. Therefore, there are $\frac{m}{2}(n+1)$ mono-indexed edges.
\end{proof}

\begin{proposition}
Let $C_m$ and $C_n$ be two cycles, where $m < n$. Then, the sparing number of $C_m+C_n$ is given by 
\[\varphi(C_m+C_n) = \left\{
\begin{array}{l l}	
\frac{m}{2}(n+2)& \quad \text{if $C_n$ is even.}\\
1+\frac{m}{2}(n+3)& \quad \text{if $C_n$ is odd.}
\end{array} \right.\]
\end{proposition}
\begin{proof}
Let $C_m$ and $C_n$ be two cycles, where $m < n$. By Theorem \ref{T-WG++}, $C_m+C_n$ is a weak IASI graph if and only if either $C_m$ or $C_n$ is $1$-uniform. Since $m < n$, let $C_m$ be $1$-uniform. 

Let $C_n$ be an even cycle. Then, $C_n$ has $\frac{n}{2}$ mono-indexed edges connecting $C_m$ and $C_n$. But, $C_n$ need not have any mono-indexed edge. Therefore, there are $m.\frac{n}{2}$ mono-indexed edges. Hence, the total number of mono-indexed edges in $C_m+C_n$ is $m+m.\frac{n}{2} = \frac{m}{2}(n+2)$.

Let $C_n$ be of odd length. Then $C_n$ has (at least) one mono-indexed edge and has $\frac{(n+1)}{2}$ mono-indexed edges connecting $C_m$ and $C_n$. Hence, the total number of mono-indexed edges is $1+m+m.\frac{n+1}{2} = 1+\frac{m}{2}(n+3)$. Therefore, there are $1+\frac{m}{2}(n+1)$ mono-indexed edges in $C_m+C_n$.
\end{proof}

In a similar way, we can establish the following result also. 

\begin{proposition}
Let $P_m$ be a path and $C_n$ be a cycle. If $m<n$, then the sparing number of $P_m+C_n$ is given by 
\[\varphi(P_m+C_n) = \left\{
\begin{array}{l l}	
\frac{m}{2}(n+2)& \quad \text{if $C_n$ is even.}\\
1+\frac{m}{2}(n+3)& \quad \text{if $C_n$ is odd.}
\end{array} \right.\]

If $m>n$, then the sparing number of $P_m+C_n$ is given by 
\[\varphi(P_m+C_n) = \left\{
\begin{array}{l l}	
\frac{n}{2}(m+2)& \quad \text{if $P_m$ is of even length.}\\
\frac{n}{2}(m+1)& \quad \text{if $P_m$ is of odd length.}
\end{array} \right.\]
\end{proposition}

\subsection{Weak IASI of the Ring sum of Graphs} 

\begin{definition}{\rm
\cite{ND} Let $G_1$ and $G_2$ be two graphs. Then the {\em ring sum} (or {\em symmetric difference}) of these graphs, denoted by $G_1\oplus G_2$, is defined as the graph with the vertex set $V_1\cup V_2$ and the edge set $E_1\oplus E_2$, leaving all isolated vertices, where $E_1\oplus E_2=(E_1\cup E_2)-(E_1\cap E_2)$.}
\end{definition}

\begin{remark}{\rm
Let $P_m$ and $P_n$ be two paths in a given graph $G$. Then, $P_m\oplus P_n$ is a path or disjoint union of paths or  a cycle. Hence, $P_m\oplus P_n$ admits a weak IASI if it is a path or disjoint union of paths or an even cycle and admits a weak IASI with at least one mono-indexed edge if it is an odd cycle.}
\end{remark}

\begin{remark}{\rm
Let $P_m$ be a path and $C_n$ be a cycle in a given graph $G$. If $P_m$ and $C_n$ are edge disjoint, then  $P_m\oplus C_n=P_m\cup C_n$.  Therefore, $P_m\oplus C_n$ admits a weak IASI if and only if $C_n$ has a weak IASI. If $P_m$ and $C_n$ have some edges in common, then $P_m\oplus C_n$ is a path. Hence, by Theorem \ref{T-SB1}, $P_m\oplus C_n$ admits a weak IASI.}
\end{remark}

The following theorem establishes the admissibility of weak IASI by the ring sum of two cycles.

\begin{theorem}\label{T-WRS}
If $C_m$ and $C_n$ are two cycles which admit weak IASIs, and $C_m\oplus C_n$ be the ring sum of $C_m$ and $C_n$. Then,
	\begin{enumerate}
	\item[(i)] if $C_m$ and $C_n$ are of same parity, $C_m\oplus C_n$ admits a weak IASI.
	\item[(ii)] if $C_m$ and $C_n$ are of different parities, $C_m\oplus C_n$ admits a weak IASI if and only if it has odd number of mono-indexed edges.
	\end{enumerate}
 \end{theorem}
\begin{proof}
Let $C_m$ and $C_n$ be two cycles which admit weak IASIs. If $C_m$ and $C_n$ have no common edges, then $C_m\oplus C_n=C_m\cup C_n$. This case has already been discussed in the previous section. Hence, assume that $C_m$ and $C_n$ have some common edges. 

Let $v_i$ and $v_j$ be the end vertices of the path common to $C_m$ and $C_n$. Let $P_r, r<m$ be the $(v_i,v_j)$-section of $C_m$ and $P_s,s<n$ be the $(v_i,v_j)$-section of $C_n$, which have no common elements other than $v_i$ and $v_j$. Hence, we have $C_m\oplus C_n=P_r\cup P_s$ is a cycle. Then, we have the following cases.

\noindent {\em Case 1:} Let $C_m$ and $C_n$ are odd cycles.
If $C_1$ and $C_2$ have an odd number of common edges, then both $P_r$ and $P_s$ are paths of even length. Hence, the cycle $P_r\cup P_s$ is an even cycle. Therefore, $C_m\oplus C_n$ has a weak IASI. If $C_m$ and $C_n$ have an even number of common edges, then both $P_r$ and $P_s$ are paths of odd length. Therefore, the cycle $P_r\cup P_s$ is an even cycle. Hence, $C_m\oplus C_n$ has a weak IASI.
	
\noindent {\em Case 2:} Let $C_m$ and $C_n$ are even cycles.
If $C_m$ and $C_n$ have an odd number of common edges, then both $P_r$ and $P_s$ are paths of odd length. Hence, the cycle $P_r\cup P_s$ is an even cycle. Therefore, $C_m\oplus C_n$ has a weak IASI. If $C_m$ and $C_n$ have an even number of common edges, then both $P_r$ and $P_s$ are paths of even length. Hence, the cycle $P_r\cup P_s$ is an even cycle. Therefore, $C_m\oplus C_n$ has a weak IASI.

\noindent {\em Case 3:} Let $C_m$ and $C_n$ be two cycles of different parities. Without loss of generality, assume that $C_m$ is an odd cycle and $C_n$ is an even cycle.
Let $C_m$ and $C_n$ have an odd number of common edges. Then, the path $P_r$ has even length and the path $P_s$ has odd length. Hence, the cycle $P_r\cup P_s$ is an odd cycle. Therefore, by Theorem \ref{T-WUOC}, $C_m\oplus C_n$ has a weak IASI if and only if $P_r\cup P_s$ has odd number of edges of set-indexing number $1$. Let $C_m$ and $C_n$ have an even number of common edges. Then, $P_r$ has odd length and $P_s$ has even length. Hence, the cycle $P_r\cup P_s$ is an odd cycle. therefore, by Theorem \ref{T-WUOC}, $C_m\oplus C_n$ has a weak IASI if and only if $P_r\cup P_s$ has odd number of edges of set-indexing number $1$.
\end{proof}

\begin{definition}{\rm
Let $H$ be a subgraph of the given graph $G$, then $G\oplus H = G-H$, which is called {\em complement of $H$ in $G$}.}
\end{definition}

Therefore, we have the following proposition on the complement of a subgraph of $G$ in $G$.

\begin{theorem}
Let $G$ be a weak IASI graph. Then, the complement of any subgraph $H$ in $G$ is also a weak IASI graph under the induced weak IASI of $G$.
\end{theorem}
\begin{proof}
Let $G$ admits a weak IASI,say $f$. Let $H$ be a subgraph of the graph $G$. The complement of $H$ in $G$, $G-H=G\oplus H$, is a subgraph of $G$. Hence, as $G$ is a weak IASI graph, by Theorem \ref{T-WSG}, the restriction of $f$ to $G-H$ is a weak IASI for $G-H$. 
\end{proof}

\subsection{Weak IASI of the Complements of Graphs}

In this section, we report some results on the admissibility of weak IASI by the complements of different weak IASI graphs and their sparing numbers. We also discuss about the sparing number of self-complementary graphs.

A graph $G$ and its comlement $\bar{G}$ have the same set of vertices and hence $G$ and $\bar{G}$ have the same set-labels for their corresponding vertices. The set-labels of the vertices in $V(G)$ under a weak IASI of $G$ need not form a weak IASI for the complement of $G$. A set-labeling of $V(G)$ that defines a weak IASI for both the graphs $G$ and its complement $\bar{G}$ may be called a {\em concurrent set-labeling}. The set-labels of the vertex set of $G$ mentioned in this section are concurrent.

\begin{proposition}\label{T-WCom-BP}
Let $G$ be a bipartite graph and let $\bar{G}$ be its complement. Then, $\bar{G}$ is a weak IASI graph if and only if $G$ and $\bar{G}$ are a $1$- uniform IASI graphs.
\end{proposition}
\begin{proof}
Let $G$ be a bipartite graph. Then, it is a weak IASI graph with bipartition of the vertex set $(X_1,X_2)$. If $G$ is $1$-uniform, then every vertex of $G$ is mono-indexed. Hence, its complement $\bar{G}$ is also $1$-uniform. Therefore, $\bar{G}$ is also a weak IASI graph. 

Conversely, assume that $\bar{G}$ is a weak IASI graph. Now, let $X_1$ be the set of all mono-indexed vertices in $G$ and if possible, let $X_2$ be the set of all vertices of $G$ having set-indexing number greater than $1$. Then, $\bar{G}$ consists of two cliques, one is the graph $G_1$ induced by $X_1$ and other is the graph $G_2$ induced by $X_2$. clearly, the $G_1$ is $1$-uniform. If $G_2$ is not $1$-uniform, then each vertex of $G_2$ have set-indexing number greater than $1$, which is a contradiction to the hypothesis that $\bar{G}$ has a weak IASI. Then, both $G_1$ and $G_2$ are $1$-uniform components of $\bar{G}$. That is, each vertex in $V(G)$ is mono-indexed. Hence, $G$ and $\bar{G}$ are $1$-uniform IASI graphs.
\end{proof}

Now, we proceed to verify the admissibility of weak IASI by the complements of cycles. As a result, we have the following theorem. 

\begin{proposition}\label{T-WCom-C}
Let $C_n$ be a cycle on $n$ vertices. Then, its complement $\bar{C_n}$ admits a weak IASI if and only if $C_n$ has at most one vertex of set-indexing number greater than $1$.  
\end{proposition}
\begin{proof}
We have $C_n\cup \bar{C_n}=K_n$. If $\bar{C_n}$ is a weak IASI graph, then by Theorem \ref{T-WUG}, $K_n$ is also a weak IASI graph. Then by Theorem \ref{T-WKN}, at most one vertex of $C_n$ can have a set-indexing number greater than $1$. Conversely, let at most one vertex of $C_n$ (and $\bar{C_n}$) has a vertex of set-indexing number greater than $1$. Then, every edge of $C_n$ and $\bar{C_n}$ has at least one end vertex that is mono-indexed. Hence, $\bar{C_n}$ is a weak IASI graph. 
\end{proof}

\begin{corollary}\label{C-WCom-C}
Let $C_n$ be a cycle on $n$ vertices. If $C_n$ and its complement $\bar{C_n}$ are weak IASI graphs, then the minimum number of mono-indexed edges in $\bar{C_n}$ is $\frac{1}{2}n(n-3)$.
\end{corollary}
\begin{proof}
If $C_n$ and its complement $\bar{C_n}$ are weak IASI graphs, then by Proposition \ref{T-WCom-C}, $C$ can have at most one vertex of set-indexing number greater than 1. That is, $C$ can have at most 2 edges that is not mono-indexed. Hence, by Theorem \ref{T-WKN}, $\bar{C}$ contains at least $\frac{1}{2}(n-1)(n-2)-2= \frac{1}{2}n(n-3)$ edges.
\end{proof}

\begin{corollary}
Let $G$ be an $r$-regular weak IASI graph. If its complement $\bar{G}$ is also a weak IASI graph, then $\bar{G}$ contains at least $\frac{1}{2}[(n-1)(n-2)-2r]$ mono-indexed edges.
\end{corollary}
\begin{proof}
Let $G$ be an $r$-regular graph. Then, its complement $\bar{G}$ admits a weak IASI if and only if $G$ can have at most one vertex that is not mono-indexed. Therefore, since $G$ is $r$-regular, it can have at most $r$ edges that are not mono-indexed. Hence, $\bar{G}$ contains at least $\frac{1}{2}(n-1)(n-2)-r = \frac{1}{2}[(n-1)(n-2)-2r]$ mono-indexed edges.
\end{proof}

\begin{proposition}
Let $G$ be a connected weak IASI graph on $n$ vertices. If its complement $\bar{G}$ is also a weak IASI graph, then $\bar{G}$ contains at least $\frac{1}{2}[(n-1)(n-2)-2r]$ mono-indexed edges, where $r=\Delta(G)$, the maximum vertex degree.
\end{proposition}
 \begin{proof}
Let $G$ be an $r$-regular graph. Let $v$ be a vertex in $G$ of degree $r=\Delta(G)$.  The complement $\bar{G}$ of $G$ admits a weak IASI if and only if $G$ can have at most one vertex that is not mono-indexed. If we label $v$ by an $r$-element set, $G$ has $r$ edges that are not mono-indexed. That is, $G$ can have at most $r$ mono-indexed edges. Hence, $\bar{G}$ contains at least $\frac{1}{2}(n-1)(n-2)-r = \frac{1}{2}[(n-1)(n-2)-2r]$ mono-indexed edges.
\end{proof}

An interesting question that arises here is about the number of mono-indexed edges in a self-complementary, weak IASI graph. The following results address this problem.

\begin{proposition}\label{T-WSCG}
If $G$ is a self-complementary $r$-regular graph on $n$ vertices which admits a weak IASI, then $G$ and $\bar{G}$ contain at least $\frac{1}{2}r(2r-1)$ mono-indexed edges. 
\end{proposition}
\begin{proof}
Since, the vertices of $G$ and $\bar{G}$ have the same set-labels and $G\cup \bar{G}=K_n$, by Theorem \ref{T-WKN}, at most one vertex of $G$ and $\bar{G}$ can have a non-singleton set-label. Label a vertex of $G$, say $v$, by a non-singleton set. Then, since $G$ is $r$-regular, $r$ edges incident on $v$ are not mono-indexed. That is, $G$ has at most $r$ edges that are not mono-indexed. Since $G$ is self-complementary, $E(G)=E(\bar{G})$ and $E(G)\cup E(\bar{G})=E(K_n)=\frac{n(n-1)}{2}$. $|E(G)|=\frac{n(n-1)}{4}$. Since $G$ and $\bar{G}$ have at most $r$ edges that are not mono-indexed, the minimum number of mono-indexed edges in $G$ is $\frac{n(n-1)}{4}-r=\frac{1}{4}[n(n-1)-4r]$.
But, since $G\cong \bar{G}$ and $G\cup \bar{G}=K_n$, degree of $v$ in $G\cup \bar{G}$ is $(n-1)$. Hence, $2r=n-1$. Therefore, the minimum number of mono-indexed edges in $G$ is $\frac{1}{4}[n(n-1)-4r] = \frac{1}{4}[(2r+1)2r-4r]=\frac{1}{2}r(2r-1)$.
\end{proof}

\begin{remark}{\rm
We note that $C_5$ is the only cycle that is self-complementary. That is, $C_5$ is the only graph that is $2$-regular and self-complementary. Hence, $C_5$ or its complement can have at least $3$ mono-indexed edges under the IASI which is a weak IASI for both of them.}
\end{remark}

If $G$ is not $r$-regular, the number of mono-indexed edges in $G$ and $\bar{G}$ need not be equal. The relation between number of mono-indexed edges in $G$ and $\bar{G}$ is given in the following proposition.
 
\begin{proposition}
If $G$ is a self-complementary graph on $n$ vertices which has $l$ mono-indexed edges, then the number of mono-indexed edges in $\bar{G}$ is $n-l-1$.
\end{proposition}

\section{Conclusion}

In this paper, we have discussed the admissibility of certain finite graph operations. More properties and characteristics of weak IASIs, both uniform and non-uniform, are yet to be investigated.  We have formulated some conditions for some graph classes and graph operations to admit weak and strong IASIs. The problems of establishing the necessary and sufficient conditions for various graphs and graph classes to have certain IASIs still remain unsettled. All these facts highlight a wide scope for further studies in this area.

\end{document}